\documentclass[12pt]{article}
\title{On the proof of elimination of imaginaries in algebraically closed valued fields}
\author{Will Johnson}

\usepackage{amsmath, amssymb, amsthm}    	
\usepackage{fullpage} 	
\usepackage{amscd}
\usepackage{hyperref}
\usepackage[all]{xy}
\usepackage{centernot}

\DeclareMathOperator*{\forkindep}{\raise0.2ex\hbox{\ooalign{\hidewidth$\vert$\hidewidth\cr\raise-0.9ex\hbox{$\smile$}}}}

\newcommand{\Sym}{\operatorname{Sym}}

\newcommand{\Hom}{\operatorname{Hom}}

\newcommand{\res}{\operatorname{res}}
\newcommand{\Aut}{\operatorname{Aut}}

\newcommand{\acl}{\operatorname{acl}}
\newcommand{\dcl}{\operatorname{dcl}}
\newcommand{\tp}{\operatorname{tp}}

\newcommand{\val}{\operatorname{val}}

\newtheorem{theorem}{Theorem}[section] 
\newtheorem{lemma}[theorem]{Lemma}
\newtheorem{lemmadefinition}[theorem]{Lemma-Definition}
\newtheorem{corollary}[theorem]{Corollary}

\newtheorem*{theorem-star}{Theorem}
\newtheorem*{conjecture-star}{Conjecture}

\theoremstyle{definition}
\newtheorem{definition}[theorem]{Definition}

\theoremstyle{remark}
\newtheorem{remark}[theorem]{Remark}
\newtheorem{claim}[theorem]{Claim}

\newtheorem{observation}[theorem]{Observation}
\newtheorem*{acknowledgment}{Acknowledgments}

\newenvironment{claimproof}[1][\proofname]
               {
                 \proof[#1]
                 
               }
               {
                 \endproof
               }

\newcommand{\Oo}{\mathcal{O}}

\begin{document}
\maketitle

ACVF is the theory of non-trivially valued \emph{a}lgebraically \emph{c}losed \emph{v}alued \emph{f}ields.  This theory is the model companion of the theory of valued fields.  ACVF does not have elimination of imaginaries in the home sort (the valued field sort).  Nevertheless, Haskell, Hrushovski, and Macpherson in \cite{HHM} were able to find a collection of ``geometric sorts'' in which elimination of imaginaries holds.

Let $K$ be a model of ACVF, with valuation ring $\mathcal{O}$ and residue field $k$.
A \emph{lattice} in $K^n$ is an $\mathcal{O}$-submodule $\Lambda \subseteq K^n$ isomorphic to $\mathcal{O}^n$.  Let $S_n$ denote the set of lattices in $K^n$.  This is an interpretable set; it can be identified with $GL_n(K)/GL_n(\mathcal{O})$.  For each lattice $\Lambda \subseteq K^n$, let $\res \Lambda$ denote $\Lambda \otimes_{\mathcal{O}} k$, a $k$-vector space of dimension $n$.  Let $T_n$ be
\[ T_n = \bigcup_{\Lambda \in S_n} \res \Lambda = \{(\Lambda, x) : \Lambda \in S_n, ~ x \in \res \Lambda\}.\]
This set is again interpretable.

The main result of \cite{HHM} is the following:
\begin{theorem-star}[Haskell, Hrushovski, Macpherson]
ACVF eliminates imaginaries relative to the sorts $K$, $\{S_n : n \ge 1\}$ and $\{T_n : n \ge 1\}$.
\end{theorem-star}

The proof in \cite{HHM} is long and technical, and we aim to give a more straightforward proof.  Our proof is a variant of Hrushovski's shorter proof in \cite{udi}, except that our strategy for coding definable types is different---see \ref{sec:contribution}.  We also give a new proof that finite sets of modules can be coded in the geometric sorts---see \ref{sec:horror}.

The proof of elimination of imaginaries given here aims to be more conceptual and less technical than previous proofs.  We prove no new results.  We include many details that are well-known at this point, for the sake of being self-contained.  

\section{Review of ACVF}
\subsection{Notation}
In a model of ACVF, $K$ is the home sort (the valued field), $\mathcal{O} \subseteq K$ is the valuation ring, $\mathfrak{M}$ is the maximal ideal in $\mathcal{O}$, $k = \mathcal{O}/\mathfrak{M}$ is the residue field, $\Gamma = K^\times/\mathcal{O}^\times$ is the valuation group, $\res : \mathcal{O} \to k$ is the residue map, and $\val : K \to \Gamma \cup \{+\infty\}$ is the valuation.  The value group is written additively, and ordered so that
\[ \mathcal{O} = \{x \in K : \val(x) \ge 0\}.\]

A \emph{lattice} in $K^n$ is an $\mathcal{O}$-submodule $\Lambda$ of $K^n$ which is free of rank $n$, i.e., isomorphic to $\mathcal{O}^n$.  If $\Lambda$ is a lattice, $\res \Lambda$ will denote $\Lambda / \mathfrak{M}\Lambda = \Lambda \otimes_{\mathcal{O}} k$.  This is always an $n$-dimensional $k$-vector space.  We will use the following interpretable sets:
\begin{itemize}
\item $S_n$, the set of lattices in $K^n$.
\item $T_n$, the set of pairs $(\Lambda,\xi)$, where $\Lambda \in S_n$ and $\xi \in \res \Lambda$
\item $R_{n,\ell}$, the set of pairs $(\Lambda,V)$, where $\Lambda \in S_n$ and $V$ is an $\ell$-dimensional subspace of $\res \Lambda$.
\end{itemize}
Each of these sets is easily interpretable in ACVF.  Our main goal will be to prove that ACVF has elimination of imaginaries in the sorts $K$ and $R_{n,\ell}$.  In \S\ref{coda}, we will note how this implies elimination of imaginaries in $K, S_n$, and $T_n$, the standard ``geometric sorts'' of \cite{HHM}.  \textbf{But until then, the term ``geometric sorts'' will mean the sorts $K$ and $R_{n,\ell}$.}

When working in an abstract model-theoretic context, the monster model will be denoted $\mathbb{M}$.  If a definable set or other entity $X$ has a code in $\mathbb{M}^{eq}$, the code will be denoted $\ulcorner X \urcorner$.  Unless stated otherwise, ``definable'' will mean ``interpretable.''

\subsection{Basic facts}
We assume without proof the following well-known facts about ACVF.  Many of these are discussed in \cite{survey}.
\begin{itemize}
\item Models of ACVF are determined up to elementary equivalence by characteristic and residue characteristic, which must be $(0,0)$, $(p,p)$, or $(0,p)$ for some prime $p$.
\item ACVF has quantifier elimination in the language with one sort $K$, with the ring structure on $K$, and with a binary predicate for the relation $\val(x) \ge \val(y)$.
\item C-minimality: Every definable subset $D \subseteq K^1$ is a boolean combination of open and closed balls (including points).  More precisely, $D$ can be written as a disjoint union of ``swiss cheeses,'' where a swiss cheese is a ball with finitely many proper subballs removed.  There is a canonical minimal way of decomposing $D$ as a disjoint union of swiss cheeses.  All the balls involved in this decomposition are algebraic over the code for $D$.
\item The theory ACVF does not have the independence property.  That is, ACVF is NIP.
\item The value group $\Gamma$ is o-minimal, in the sense that every definable subset of $\Gamma^1$ is a finite union of points and intervals with endpoints in $\Gamma \cup \{\pm \infty\}$.  (In fact, $\Gamma$ is a stably embedded pure divisible ordered abelian group.)
\item The residue field $k$ is strongly minimal, hence stable and stably embedded.  Moreover, every definable subset of $k^n$ is coded by a tuple from $k$.  (In fact, $k$ is a stably embedded pure algebraically closed field.)
\end{itemize}
The first two points are due to Robinson \cite{Robinson}, and the third is due to Holly \cite{holly}.
The last three points are easy consequences of C-minimality\footnote{The final point uses the following fact: if $T$ is a strongly minimal theory, in which $\acl(\emptyset)$ is infinite and finite sets of tuples are coded by tuples, then $T$ eliminates imaginaries.  This is Lemma 1.6 in \cite{anand}.}, and the last two can also be seen from the quantifier elimination result in the three-sorted language discussed in \cite{survey} and \cite{HHM}.

\subsection{Valued $K$-vector spaces}\label{val-k-vec}
Let $K$ be an \emph{arbitrary} valued field.  Following Section 2.5 of \cite{udi},
\begin{definition}
A \emph{valued $K$-vector space} is a $K$-vector space $V$ and a set $\Gamma(V)$ together with the following structure:
\begin{itemize}
\item A total ordering on $\Gamma(V)$
\item An action
\[ + : \Gamma(K) \times \Gamma(V) \to \Gamma(V)\]
of $\Gamma(K) = \Gamma$ on $\Gamma(V)$, strictly order-preserving in each variable (hence free)
\item A surjective map $\val : V \setminus \{0\} \to \Gamma(V)$, such that
\[ \val(w+v) \ge \min(\val(w),\val(v))\]
\[ \val(\alpha \cdot v) = \val(\alpha) + \val(v)\]
for $w, v \in V$ and $\alpha \in K$, with the usual convention that $\val(0) = +\infty > \Gamma(V)$.
\end{itemize}
\end{definition}
This is merely a variation on the notion of a normed vector space over a field with an absolute value.

\begin{remark}\label{few-orbits}
If $\dim_K V$ is finite, then the action of $\Gamma(K)$ on $\Gamma(V)$ has finitely many orbits.  In fact,
\[ |\Gamma(V)/\Gamma(K)| \le \dim_K V.\]
\end{remark}
\begin{proof}
Let $v_1, \ldots, v_n$ be non-zero vectors with $\val(v_n)$ in different orbits of $\Gamma(K)$.  We will show that the $v_i$ are linearly independent.  If not, let $w_1, \ldots, w_m$ be a minimal subset which is linearly dependent.  Then $\sum_i x_i w_i = 0$ for some $x_i \in K^\times$.  But by assumption,
\[ \val(x_i w_i ) = \val(x_i) + \val(w_i) \ne \val(x_j) + \val(w_j) = \val(x_j w_j)\]
for any $i \ne j$.  By the ultrametric inequality in $V$, $\sum_i x_i w_i$ cannot be zero, a contradiction.
\end{proof}

For the rest of this section, we will assume that all valued $K$-vector spaces $V$ have value group $\Gamma(V) = \Gamma(K) = \Gamma$, since the goal is Theorem~\ref{spherical-completeness}.

If $V$ and $W$ are two such valued $K$ vector spaces, we can form a ``direct sum'' $V \oplus W$ by setting
\[ \val(v,w) = \min(\val(v),\val(w)).\]
For example, $K^{\oplus n}$ is a valued $K$-vector space with underlying vector space $K^n$, with value group $\Gamma(K)$, and with valuation map given by
\[ \val(x_1, \ldots, x_n) = \min(\val(x_1),\ldots,\val(x_n)).\]
If $V$ and $W$ are two subspaces of a valued $K$-vector space, say that $V$ and $W$ are \emph{perpendicular} if $V \cap W = \emptyset$ and $V + W$ is isomorphic to $V \oplus W$.  In other words, $V$ and $W$ are perpendicular if $\val(v + w) = \min(\val(v),\val(w))$ for every $v \in V$ and $w \in W$.

Recall that a valued field $K$ is \emph{spherically complete} if every descending sequence of balls in $K$ has non-empty intersection.  If $V$ is a valued $K$-vector space, a \emph{ball} in $V$ is a set of the form
\[ \{\val(x - a) \ge \gamma\} \text{ or } \{\val(x - a) > \gamma\}\]
for $a \in V$ and $\gamma \in \Gamma(V)$.  We say that $V$ is \emph{spherically complete} if every descending sequence of balls in $V$ has a non-empty intersection.
\begin{remark}~ \label{steps}
\begin{enumerate}
\item If $V$ and $W$ are spherically complete, so is $V \oplus W$, because the balls in $V \oplus W$ are of the form $B_1 \times B_2$, with $B_1$ a ball in $V$ and $B_2$ a ball in $W$.
\item If $V$ is a subspace of a valued $K$-vector space $W$, and $a \in W$, then the intersection of any ball in $W$ with $a + V$ is either empty or a ball in the affine subspace $a + V$.
\item If $V$ is a spherically complete subspace of $W$ and $a \in W$, and $\mathcal{F}$ is the family of closed balls in $W$ centered at the origin which intersect $a + V$, then $\mathcal{F} \cap V$ is a nested chain of balls in $a + V$, so it has a non-empty intersection.  Equivalently, the following set has a maximum:
\[ \{ \val(a + v) : v \in V\}.\]
That is, some element of $a + V$ is maximally close to 0.
\end{enumerate}
\end{remark}
\begin{lemma}\label{steps4}
Let $W$ be a valued $K$-vector space.  Let $V$ be a subspace.  Suppose that $a \in W \setminus V$ is maximally close to 0 among elements of $a + V$.  Then $K \cdot a$ is perpendicular to $V$.
\end{lemma}
\begin{proof}
We need to show that
\begin{equation} \val(v + \alpha a) = \min(\val(v),\val(\alpha a)) \label{meh}\end{equation}
for $v \in V$ and $\alpha \in K$.  Replacing $v$ and $\alpha$ with $\alpha^{-1}v$ and $\alpha^{-1}\alpha$ changes both sides of (\ref{meh}) by the same amount, so we may assume that $\alpha = 0$ or $\alpha = 1$.

The $\alpha = 0$ case is trivial.  Suppose that $\alpha = 1$; we want to show $\val(v + a) = \min(\val(v),\val(a))$.  If $\val(v) \ne \val(a)$, then $\val(v + a) = \min(\val(v),\val(a))$ by the ultrametric inequality.  In the case where $\val(v) = \val(a)$, the ultrametric inequality only implies
\begin{equation} \val(v + a) \ge \min(\val(v),\val(a)) = \val(a).\label{eq-7}\end{equation}
But $\val(v + a) \le \val(a)$, by the assumption on $a$.  So equality holds in (\ref{eq-7}).
\end{proof}

\begin{theorem}\label{spherical-completeness}
Suppose $K$ is spherically complete, $V$ is an $n$-dimensional $K$-vector space, and $\Gamma(K) = \Gamma(V)$.  Then $V$ is isomorphic to $K^{\oplus n}$.  In other words, there is a basis $\{v_1, \ldots, v_n\} \subseteq V$ such that
\[ \val(x_1 v_1 + \cdots + x_n v_n) = \min_{1 \le i \le n} \val(x_i) \text{ for every $\vec{x} \in K^n$}.\]
\end{theorem}
In \cite{udi}, Hrushovski calls $\{v_1, \ldots, v_n\}$ a ``separating basis.''
\begin{proof}
Proceed by induction on $\dim_K V$.  The one-dimensional case is easy.  Let $V'$ be a codimension 1 subspace.  By induction, $V'$ is isomorphic to $K^{\oplus (n-1)}$, so $V'$ is spherically complete.  Choose some $a_0 \in V \setminus V'$ and let $a$ be an element of $a_0 + V'$ maximally close to 0.  By Lemma~\ref{steps4}, $K \cdot a$ is perpendicular to $V'$.  Thus $V \cong V' \oplus K \cong K^{\oplus (n-1)} \oplus K = K^{\oplus n}$.
\end{proof}%

\subsection{Definable submodules of $K^n$}\label{module-classify}
We now return to the setting of ACVF.

Recall that every model of ACVF is elementarily equivalent to a spherical complete one.\footnote{This is well-known, and discussed in \cite{survey}.   In the pure characteristic case, one can use fields of Hahn series.  In the mixed characteristic case, one can use metric ultrapowers of $\mathbb{C}_p$.}

\begin{theorem}\label{module-classification}
Let $K$ be a model of ACVF.
Let $V$ be a definable $K$-vector space, with $\dim_K V < \infty$.  Let $N \subseteq V$ be a \emph{definable} $\mathcal{O}$-submodule.  Then $N$ is isomorphic to $K^{n_1} \times \mathcal{O}^{n_2} \times \mathfrak{M}^{n_3}$ for some $n_1, n_2, n_3 < n$.
\end{theorem}
\begin{proof}
We are trying to prove a conjunction of first-order sentences, so we may replace $K$ with an elementarily equivalent model.  Therefore, we may assume $K$ is spherically complete.

Replacing $V$ with the $K$-span of $N$, we may assume that $V$ is the $K$-span of $N$.  Similarly, if $W$ denotes the largest $K$-vector space contained in $N$, then by quotienting out $W$, we may assume that $N$ contains no nontrivial $K$-vector spaces.  Now $\bigcup_{\alpha \in K^\times} \alpha N = V$ and $\bigcap_{\alpha \in K^\times} \alpha N = 0$. 

For any nonzero $v \in V$, let
\[ \val(v) = \sup \{\val(\alpha) : v \in \alpha N\} = \inf \{\val(\alpha) : v \notin \alpha N\}.\]
This is well-defined by o-minimality of $\Gamma$, and one easily checks that
\begin{equation} \val(\beta v) = \val(\beta) + \val(v).\label{huh0}\end{equation}
\begin{equation} \val(v) > 0 \implies v \in N \label{huh1} \end{equation}
\begin{equation} \val(v) < 0 \implies v \notin N \label{huh2} \end{equation}
for all $\beta \in K$, $v \in V$.  We claim that $\val : V \to \Gamma$ makes $V$ into a valued $K$-vector space.  Given (\ref{huh0}), we merely need to check the ultrametric inequality
\[ \val(v + w) \ge \min(\val(v),\val(w)).\]
If this failed, then multiplying everything by an appropriate scalar, we would get
\[ \val(v + w) < 0 < \min(\val(v),\val(w)).\]
But then $v, w \in N$ and $v + w \notin N$, contradicting the fact that $N$ is a module.

So $\val : V \to \Gamma$ makes $V$ into a valued $K$-vector space.  By Theorem~\ref{spherical-completeness}, we can assume that $V$ is $K^{\oplus n}$.  Then (\ref{huh1}-\ref{huh2}) mean the following for $\vec{x} \in K^n$:
\begin{itemize}
\item If $\val(x_i) > 0$ for every $i$, then $\vec{x} \in N$.  In other words, $\mathfrak{M}^n \subseteq N$.
\item If $\val(x_i) < 0$ for some $i$, then $\vec{x} \notin N$.  In other words, $N \subseteq \mathcal{O}^n$.
\end{itemize}
So $N$ is sandwiched between $\mathcal{O}^n$ and $\mathfrak{M}^n$.  But the possibilities for $N$ then correspond to the submodules of $\mathcal{O}^n/\mathfrak{M}^n$, i.e., the $k$-subspaces of $k^n$.  These are easy to deal with.

Specifically, note that $N/\mathfrak{M}^n$ is a $k$-subspace of $\mathcal{O}^n/\mathfrak{M}^n = k^n$.  Let $\gamma$ be an element of $GL_n(k)$ sending $N/\mathfrak{M}^n \subseteq k^n$ to $k^\ell \times 0^{n - \ell} \subseteq k^n$ for $\ell = \dim_k N/\mathfrak{M}^n$.  Then $\gamma$ can be lifted to some $\gamma' \in GL_n(\mathcal{O})$, because $\mathcal{O}$ is a local ring.  If $N' = \gamma'(N)$, then $N'/\mathfrak{M}^n$ is $k^\ell \times 0^{n- \ell} = (\mathcal{O}^\ell \times \mathfrak{M}^{n-\ell})/\mathfrak{M}^n$.  So $N' = \mathcal{O}^\ell \times \mathfrak{M}^{n- \ell} \subseteq K^n$.  But $N'$ and $N$ are isomorphic.
\end{proof}

Let $Mod_n$ denote the set of definable submodules of $K^n$.  The theorem implies that the elements of $Mod_n$ fall into finitely many definable families.  Consequently, we get the following
\begin{corollary}
The set $Mod_n$ is interpretable.
\end{corollary}

\section{Generalities on Definable Types}
Work in an arbitrary theory $T$, with monster model $\mathbb{M}$.  By ``$C$-definable type,'' we will mean $C$-definable type over the monster, as opposed to some smaller model, unless stated otherwise.  By ``definable type,'' we mean a $C$-definable type for some $C \subseteq \mathbb{M}$.

In this section we review some well-known facts about definable types.  We omit many of the proofs, which are usually straightforward.

\subsection{Operations on definable types}
If $p$ is a $C$-definable type and $f$ is a $C$-definable function, there is a unique $C$-definable type $f_* p$ which is characterized by the following property:
\[ a \models p|B \implies f(a) \models f_*p | B \text{, for all small $B \supseteq C$ and all $a$.}\]
The choice of $C$ does not matter---if $p$ and $f$ are $C'$-definable for some other set $C'$, then the resulting $f_* p$ is the same.  The type $f_* p$ is called the \emph{pushforward} of $p$ along $f$.

If $p$ and $q$ are two $C$-definable types, there is a unique $C$-definable type $p \otimes q$ which is characterized by the following property:
\[ (a,b) \models p \otimes q | B \iff (a \models p | Bb) \wedge (b \models q | B) \text{, for all small $B \supseteq C$ and all $a,b$.}\]
Again, $p(x) \otimes q(y)$ does not depend on the choice of $C$.  The product operation is associative:
\[ (p(x) \otimes q(y)) \otimes r(z) = p(x) \otimes (q(y) \otimes r(z)),\]
but commutativity
\[ p(x) \otimes q(y) \stackrel{?}{=} q(y) \otimes p(x)\]
can fail.

\begin{remark}\label{products-and-pushforwards}
If $f,g$ are definable functions and $p, q$ are definable types, then $f_*p \otimes g_*q = (f \times g)_*(p \otimes q)$, where $f \times g$ sends $(x,y)$ to $(f(x),g(y))$.
\end{remark}

\subsection{Generically stable types}
Now assume that $T$ is NIP.  (This includes the case of ACVF.)
\begin{definition}
A definable type $p(x)$ is \emph{generically stable} if $p(x) \otimes q(y) = q(y) \otimes p(x)$ for every definable type $q(y)$.
\end{definition}
For other equivalent definitions of generic stability, see Section 3 of \cite{udi-anand}.

\begin{definition}\label{domination-def}
Let $f$ be a $C$-definable function and $p$ be a $C$-definable type.  Abusing terminology significantly, say that $p$ is \emph{dominated along $f$} if
\[ f(a) \models f_*p | B \implies a \models p|B \text{ for all small $B \supseteq C$ and all $a$.}\]
\end{definition}
Note that the converse implication holds by definition of $f_* p$. Unlike the previous definitions, this \emph{does} depend on the choice of $C$.  In the cases we care about, $C$ will be $\emptyset$.
\begin{remark}
Suppose $p$ is dominated along $f$, and $q$ is some other definable type.  If $B$ is a set over which everything is defined and over which the domination holds, then
\begin{equation}
(f(a),b) \models f_*p \otimes q | B \implies (a,b) \models p \otimes q | B
\label{convoluted}
\end{equation}
\end{remark}

We will use the following basic facts about generically stable types:
\begin{theorem} \label{gs} ~
\begin{description}
\item[(a)] Products of generically stable types are generically stable.
\item[(b)] Pushforwards of generically stable types are generically stable.
\item[(c)] If $p$ is dominated along $f$ and $f_* p$ is generically stable, then $p$ is generically stable.
\item[(d)] If $p$ and $q$ are generically stable, dominated along $f$ and $g$, respectively, then $p \otimes q$ is dominated along $f \times g$.
\item[(e)] To check generic stability, it suffices to show that $p$ commutes with itself, i.e., $p(x_1) \otimes p(x_2) = p(x_2) \otimes p(x_1)$.
\end{description}
\end{theorem}
\begin{proof} ~
\begin{description}
\item[(a)] If $p$ and $q$ are generically stable, and $r$ is arbitrary, then $p \otimes q \otimes r = p \otimes r \otimes q = r \otimes p \otimes q$.
\item[(b)] Suppose $p$ is generically stable, $f$ is a definable function, and $q$ is arbitrary.  Then $p(x) \otimes q(y) = q(y) \otimes p(x)$.  Pushing both sides forwards along $(f \times id)$ and applying Remark~\ref{products-and-pushforwards}, we get that $f_* p(x') \otimes q(y) = q(y) \otimes f_*p(x')$.
\item[(c)] 
Let $q$ be another invariant type; we will show that $p(x) \otimes q(y) = q(y) \otimes p(x)$.  Let $B$ be a set over which $p, q, f$ are defined.  Let $(b,a)$ realize $q \otimes p | B$.  By Remark~\ref{products-and-pushforwards}, $(b,f(a)) \models q \otimes f_*p | B$.  Since $f_*p $ is generically stable, $(f(a),b) \models f_*p \otimes q | B$.  By (\ref{convoluted}), $(a,b) \models p \otimes q | B$.  So $p \otimes q$ and $q \otimes p$ agree when restricted to the arbitrary set $B$.
\item[(d)]  Let $B$ be a sufficiently big set.  Suppose that $(f(a),g(b)) \models f_*p \otimes g_*q | B$.  We need to show that $(a,b) \models p \otimes q | B$.  By (\ref{convoluted}), $(a,g(b)) \models p \otimes g_*q | B$.  By generic stability of $p$, $(g(b),a) \models g_*q \otimes p | B$.  By (\ref{convoluted}) again, $(b,a) \models q \otimes p | B$.  By generic stability again, $(a,b) \models p \otimes q | B$.
\item[(e)] Suppose $p(x)$ commutes with itself, but $p(x) \otimes q(y) \ne q(y) \otimes p(x)$.  Choose some formula $\phi(x;y;c)$ which is in $p(x) \otimes q(y)$ and not in $q(y) \otimes p(x)$.  We will prove that $\phi(x;y,z)$ has the independence property.  Let $n$ be arbitrary.  Let $a_1, \ldots, a_n, b, a_{n+1}, \ldots, a_{2n}$ realize $p^{\otimes n} \otimes q \otimes p^{\otimes n}$ restricted to $c$.  Then $\models \phi(a_i;b;c) \iff i \le n$, by choice of $\phi(x;y;c)$.  The fact that $p$ commutes with itself implies that all permutations of $(a_1,\ldots,a_{2n})$ have the same type over $c$.  Therefore, for each permutation $\pi$ of $\{1, \ldots, 2n\}$, we can find a $b_\pi$ such that $\phi(a_i;b_\pi;c)$ holds iff $\pi(i) \le n$.  It follows that for any $S \subseteq \{1,\ldots,n\}$, we can find a $b_S$ such that $\phi(a_i;b_S;c)$ holds if and only if $i \in S$.  As $n$ was arbitrary, $T$ has the independence property, a contradiction.
\end{description}
\end{proof}

\subsection{Definable types in ACVF}\label{our-def-types}
Now work in ACVF.  Recall that ACVF is NIP.  We will make use of several definable types:
\begin{itemize}
\item If $B$ is an open or closed ball in the home sort, then there is a complete type $p_B(x)$ over $\mathbb{M}$ which says that $x \in B$ and $x$ is not in any strictly smaller balls.  This type is called the \emph{generic type of $B$}.  Completeness follows from $C$-minimality.  This type is definable, essentially because if $B'$ is any other ball, then the formula $x \in B'$ is in $p_B(x)$ if and only if $B' \supseteq B$.  If $C$ is any set of parameters over which $B$ is defined, then $p_B|C$ says precisely that $x$ is in $B$, and $x$ is not in any $\acl^{eq}(C)$-definable proper subball of $B$.
\item There is also a type $p_k(x)$ which says that $x$ is in the residue field, and is not algebraic over $\mathbb{M}$.  This is called the \emph{generic type of the residue field}, and is definable because $k$ is strongly minimal.  If $C$ is any set of parameters, $p_k|C$ says precisely that $x \in k$ and $x \notin \acl^{eq}(C)$.
\item The valuation ring $\mathcal{O}$ is a closed ball, so it has a generic type $p_\mathcal{O}$.  Over any set of parameters $C$, $p_\mathcal{O}(x)$ says that $x \in \mathcal{O}$, and that $x$ is not in any $\acl^{eq}(C)$-definable proper subballs of $\mathcal{O}$.  Every proper subball of $\mathcal{O}$ is contained in a unique one of the form $\res^{-1}(\alpha)$, for $\alpha \in k$.  Consequently, $p_\mathcal{O}|C$ equivalently says that $x \in \mathcal{O}$ and that $x \notin \res^{-1}(\alpha)$ for any $\alpha \in \acl^{eq}(C)$.  Equivalently,
\[ x \models p_\mathcal{O}|C \iff \res(x) \models p_k|C\]
Therefore, $p_\mathcal{O}$ is dominated along $\res$, and $\res_* p_\mathcal{O} = p_k$.
\end{itemize}
The type $p_k$ is generically stable.  To see this, use (e) of Theorem~\ref{gs} and stability of $k$.

Since $p_k$ is generically stable, so is $p_\mathcal{O}$, by Theorem~\ref{gs}(c).  If $B$ is any closed ball, then there is an affine transformation $f(x) = ax + b$ sending $\mathcal{O}$ to $B$, and $p_B = f_* p_\mathcal{O}$.  By Theorem~\ref{gs}(b), each $p_B$ is generically stable.

Let $p_{\mathcal{O}^n}$ be $p_\mathcal{O}^{\otimes n}$.  We think of $p_{\mathcal{O}^n}$ as the generic type of the lattice $\mathcal{O}^n$.
  By Theorem~\ref{gs}, $p_{\mathcal{O}^n}$ is generically stable, and is dominated along the map $(x_1,\ldots,x_n) \mapsto (\res(x_1),\ldots,\res(x_n))$.  Also, the pushforward along this map is  $p_k^{\otimes n}$, the generic type of $k^n$.

The generic type of $k^n$ is stabilized by the action of $GL_n(k)$, so by domination, the generic type of $\mathcal{O}^n$ is stabilized by $GL_n(\mathcal{O})$.  In light of this, the following definition does not depend on the choice of $g$:
\begin{definition}
Let $\Lambda$ be a lattice in $K^n$.  The \emph{generic type $p_\Lambda$ of $\Lambda$} is $g_* p_{\mathcal{O}^n}$, where $g : K^n \to K^n$ is any linear map sending $\mathcal{O}^n$ to $\Lambda$.
\end{definition}
Moreover, $p_\Lambda$ is $\ulcorner \Lambda \urcorner$-definable.
Note that $p_\Lambda$ is a generically stable type, because it is a pushforward of a generically stable type.

\subsection{Left transitivity}
Now return to an arbitrary theory $T$.  Work in $T^{eq}$, so that $\acl$ and $\dcl$ mean $\acl^{eq}$ and $\dcl^{eq}$.

\begin{lemma}\label{basic-transitivity}
Suppose $C \subseteq B$ are small sets and $a_1, a_2$ are tuples (possibly infinite, but small).  If $\tp(a_1/B)$ is $C$-definable and $\tp(a_2/Ba_1)$ is $Ca_1$-definable, then $\tp(a_2a_1/B)$ is $C$-definable.
\end{lemma}
\begin{proof}
Naming the parameters from $C$, we may assume $C = \emptyset$.  Let $\phi(x_2,x_1;y)$ be a formula; we must produce a $\phi$-definition (over $\emptyset$) for $\tp(a_2a_1/B)$.  Since $\tp(a_2/Ba_1)$ is $a_1$-definable, the $\phi(x_2;x_1,y)$-type of $a_2$ over $Ba_1$ has a definition $\psi(y,a_1)$.  In particular, for every tuple $b$ from $B$,
\[ \models \phi(a_2,a_1,b) \leftrightarrow \psi(b,a_1).\]
Meanwhile, since $\tp(a_1/B)$ is 0-definable, there is a formula $\chi(y)$ such that for every $b$ in $B$,
\[ \models \psi(b,a_1) \leftrightarrow \chi(b).\]
Thus, for every $b$ in $B$,
\[ \models \phi(a_2,a_1,b) \leftrightarrow \chi(b),\]
so $\chi(y)$ is the $\phi(x_2,x_1;y)$-definition of $\tp(a_2a_1/B)$.
\end{proof}

\begin{lemmadefinition}\label{indep1}
The following are equivalent for $A, B, C$ small sets of parameters.
\begin{enumerate}
\item The $*$-type $\tp(A/BC)$ has a global $C$-definable extension.
\item For every small set of parameters $D$, there is a $C$-definable extension of $\tp(A/BC)$ to a $*$-type over $BCD$.
\item \label{tres} For every small set of parameters $D$, there is $D' \equiv_{BC} D$ such that $\tp(A/BCD')$ is $C$-definable.
\item For some small model $M \supseteq BC$, $\tp(A/M)$ is $C$-definable.
\end{enumerate}
We denote these equivalent conditions by $A \forkindep^{\textrm{def}}_C B$.
\end{lemmadefinition}
\begin{proof}
\begin{description}
\item[(1$\implies$2)] The restriction of a global $C$-definable type to $BCD$ is $C$-definable.
\item[(2$\implies$3)] Given $D$, (2) implies that there is $A' \equiv_{BC} A$ such that $\tp(A'/BCD)$ is $C$-definable.  Choose $D'$ such that $A'D \equiv_{BC} AD'$.  Then $\tp(A/BCD')$ is $C$-definable and $D' \equiv_{BC} D$.
\item[(3$\implies$4)] Applying (3) to $D$ a small model containing $BC$, we get a small model $D'$ containing $BC$ such that $\tp(A/BCD') = \tp(A/D')$ is $C$-definable.
\item[(4$\implies$1)] $C$-definable types over models have unique $C$-definable extensions to elementary extensions.  This is true even for $*$-types.
\end{description}
\end{proof}

\begin{lemma}\label{indep1-trans}
If $a_1 \forkindep^{\textrm{def}}_C B$ and $a_2 \forkindep^{\textrm{def}}_{Ca_1} B$ then $a_2 a_1 \forkindep^{\textrm{def}}_C B$, so $\forkindep^{\textrm{def}}$ satisfies left-transitivity.
\end{lemma}
\begin{proof}
We use condition (\ref{tres}) of Lemma-Definition~\ref{indep1}.  Let $D$ be a small set of parameters.  Since $a_1 \forkindep^{\textrm{def}}_C B$, there is $D' \equiv_{BC} D$ such that $\tp(a_1/BCD')$ is $C$-definable.  As $a_2 \forkindep^{\textrm{def}}_{Ca_1} B$ there is $D'' \equiv_{Ca_1B} D'$ such that $\tp(a_2/BCa_1D'')$ is $Ca_1$-definable.  Note that $\tp(a_1/BCD'')$ is $C$-definable.  By Lemma~\ref{basic-transitivity}, it follows that $\tp(a_2a_1/BCD'')$ is $C$-definable.  Since $D'' \equiv_{BC} D' \equiv_{BC} D$, we have verified $a_2 a_1 \forkindep^{\textrm{def}}_C B$ using condition (\ref{tres}).
\end{proof}

\begin{definition}
If $A, B, C$ are small sets of parameters, we will write $A
\forkindep^{\textrm{adef}}_C B$ to mean $A
\forkindep^{\textrm{def}}_{\acl(C)} B$.  (Recall that $\acl(C)$ means
$\acl^{eq}(C)$.)
\end{definition}
In other words, $a \forkindep^{\textrm{adef}}_C B$ if $\tp(a/BC)$ can be extended to a type which is almost $C$-definable, that is, $\acl(C)$-definable.  In a stable theory, $\forkindep^{\textrm{adef}}$ is exactly nonforking independence.

\begin{lemma}\label{models-acl}
If $A \forkindep^{\textrm{adef}}_C B$, then $\acl(AC) \forkindep^{\textrm{adef}}_C B$.
\end{lemma}
\begin{proof}
Replacing $C$ with $\acl(C)$, we may assume that $C = \acl(C)$.  By condition (4) of Lemma-Definition~\ref{indep1}, there is a small model $M$ containing $BC$ such that $\tp(A/M)$ is $C$-definable.  We need to show that $\tp(\acl(AC)/M)$ is $C$-definable.  This is equivalent to showing that for each $\acl(AC)$-definable set $X$, there is some $C$-definable set $X'$ such that $X \cap M = X' \cap M$.

Given such an $X$, let $X_1, \ldots, X_n$ be the conjugates of $X$ over $AC$.  Let $E$ be the equivalence relation
\[ x E y \leftrightarrow \bigwedge_{i = 1}^n (x \in X_i \leftrightarrow y \in X_i)\]
Then $E$ is $AC$-definable.  Since $\tp(AC/M)$ is $C$-definable, the restriction $E' = E \cap M$ of $E$ to $M$ is $C$-definable.  Since $E$ has finitely many equivalence classes, so does $E'$, and hence each equivalence class of $E'$ is $C$-definable, as $\acl(C) = C$.  But $X \cap M$ is a union of finitely many $E'$-equivalence classes, so $X \cap M$ is $C$-definable.
\end{proof}

\begin{lemma}\label{indep2-trans}
If $a_1 \forkindep^{\textrm{adef}}_C B$ and $a_2 \forkindep^{\textrm{adef}}_{Ca_1} B$, then $a_2a_1 \forkindep^{\textrm{adef}}_C B$, so $\forkindep^{\textrm{adef}}$ satisfies left-transitivity.
\end{lemma}
\begin{proof}
By Lemma~\ref{models-acl}, we know that $\acl(a_1 C) \forkindep^{\textrm{def}}_{\acl(C)} B$.  We are given $a_2 \forkindep^{\textrm{def}}_{\acl(Ca_1)} B$.  Combining these using Lemma~\ref{indep1-trans}, we conclude that $a_2 \acl(Ca_1) \forkindep^{\textrm{def}}_{\acl(C)} B$.  This easily implies $a_2a_1 \forkindep^{\textrm{def}}_{\acl(C)} B$, as desired.
\end{proof}

\section{An Abstract Criterion for Elimination of Imaginaries}
We state a sufficient condition for a theory $T$ to have elimination of imaginaries, extracted from \cite{udi}.
\begin{theorem}\label{ei-criterion}
Let $T$ be a theory, with home sort $K$ (meaning $\mathbb{M}^{eq} = \dcl^{eq}(K)$).  Let $\mathcal{G}$ be some collection of sorts.  If the following conditions all hold, then $T$ has elimination of imaginaries in the sorts $\mathcal{G}$.
\begin{itemize}
\item For every non-empty definable set $X \subseteq K^1$, there is an $\acl^{eq}(\ulcorner X \urcorner)$-definable type in $X$.
\item Every definable type in $K^n$ has a code in $\mathcal{G}$ (possibly infinite).  That is, if $p$ is any (global) definable type in $K^n$, then the set $\ulcorner p \urcorner$ of codes of the definitions of $p$ is interdefinable with some (possibly infinite) tuple from $\mathcal{G}$.
\item Every finite set of finite tuples from $\mathcal{G}$ has a code in $\mathcal{G}$.  That is, if $S$ is a finite set of finite tuples from $\mathcal{G}$, then $\ulcorner S \urcorner$ is interdefinable with a tuple from $\mathcal{G}$.
\end{itemize}
\end{theorem}
\begin{proof} Assume the three conditions.
\begin{claim}
For every non-empty definable set $X \subseteq K^n$, there is an $\acl^{eq}(\ulcorner X \urcorner)$-definable type in $X$.
\end{claim}
\begin{claimproof}
We proceed by induction on $n$, the base case $n = 1$ being given.  Suppose $n > 1$.  Take $X \subseteq K^n$.  Let $C = \ulcorner X \urcorner$.  Let $\pi : K^n \twoheadrightarrow K^{n-1}$ be the projection onto the first $n - 1$ coordinates.  Then $\pi(X)$ is $C$-definable, so by induction, there is an $\acl^{eq}(C)$-definable type in $\pi(X)$.  Let $a_1$ realize this type.  Then $a_1 \forkindep^{\textrm{adef}}_C C$.

Let $Y = \{y \in K^1 | (a_1,y) \in X\}$, so $Y$ is essentially $X \cap \pi^{-1}(a_1)$.  Then $Y$ is $Ca_1$-definable and non-empty.  By assumption, there is an $\acl^{eq}(Ca_1)$-definable type in $Y$.  Let $a_2$ realize this type; then $a_2 \in Y$ and $a_2 \forkindep^{\textrm{adef}}_{Ca_1} C$.  Since $a_1 \forkindep^{\textrm{adef}}_C C$, it follows that $a_1a_2 \forkindep^{\textrm{adef}}_C C$ by Lemma~\ref{indep2-trans}.  By definition of $\forkindep^{\textrm{adef}}$, there is an $\acl^{eq}(C)$-definable type $p(x_1,x_2)$ such that $a_1a_2 \models p | \acl^{eq}(C)$.  As $a_2 \in Y$, the tuple $a_1a_2$ is in $X$, so $p$ is an $\acl^{eq}(C)$-definable type in $X$.
\end{claimproof}

Let $e$ be any imaginary.  Then there is some $n$ and some 0-definable equivalence relation $E$ on $K^n$ such that $e$ is a code for some $E$-equivalence class $X$.  By the claim, there is an $\acl^{eq}(e)$-definable type $p$ in $X$.  Then $e \in \dcl^{eq}(\ulcorner p \urcorner)$, because $X$ is the unique $E$-equivalence class in which the type $\ulcorner p \urcorner$ lives.  By the second assumption, there is some small tuple $t_0 \subseteq \mathcal{G}$ such that $\ulcorner p \urcorner$ is interdefinable with $t_0$.  Thus $e \in \dcl^{eq}(t_0)$ and $t_0 \in \acl^{eq}(e)$.  By compactness, we can find some finite tuple $t$ from $\mathcal{G}$ such that $e \in \dcl^{eq}(t)$ and $t \in \acl^{eq}(e)$.  Write $e$ as $f(t)$ for some 0-definable function $f$.  Let $S$ be the (finite) set of conjugates of $t$ over $e$.  Then $S$ is $e$-definable.  Moreover, $f(t') = e$ for any $t' \in S$, so $e$ is $\ulcorner S \urcorner$-definable.  Hence $e$ and $\ulcorner S \urcorner$ are interdefinable.  By the third hypothesis, $\ulcorner S \urcorner$ has a code in $\mathcal{G}$.  So $e$ has a code in $\mathcal{G}$.  As $e$ was arbitrary, $T$ has elimination of imaginaries down to the sorts in $\mathcal{G}$.
\end{proof}

The conditions in the theorem are sufficient but not necessary for elimination of imaginaries to hold.  Namely, the first condition has nothing to do with $\mathcal{G}$, and happens to fail in $\mathbb{Q}_p$, even if $\mathcal{G}$ is chosen to be all of $\mathbb{Q}_p^{eq}$.

\subsection{Examples}
We sketch how to use Theorem~\ref{ei-criterion} to verify elimination of imaginaries in ACF and DCF$_0$ in the home sort $K$ (so $\mathcal{G}$ is merely $\{K\}$).  The first condition follows from stability: if $X$ is any non-empty definable set, then the formula $x \in X$ does not fork over $\ulcorner X \urcorner$.  If $p$ is a global type which does not fork over $\ulcorner X \urcorner$ and contains this formula, then $p$ is an $\acl^{eq}(\ulcorner X \urcorner)$-definable type in $X$.

For the second condition, one must check that every type has a code (possibly infinite) in the home sort.  If $p$ is a type in $K^n$, then there is a minimal Zariski-closed or Kolchin-closed set $V$ containing $p$, and $p$ and $V$ have the same code.  The second condition thus reduces to coding Zariski-closed sets or Kolchin-closed sets, respectively.  So does the third condition, since any finite subset of $K^n$ is Zariski-closed and Kolchin-closed.  Now, to code a Zariski-closed or Kolchin-closed set $V$, we merely need to code the ideal $I$ of polynomials or differential polynomials which vanish on $V$.  In the ACF case, this reduces to coding, for each $d < \omega$, the intersection of $I$ with the space of degree $\le d$ polynomials in $K[X_1,\ldots,X_n]$.  Something similar happens in DCF.  So the problem reduces to coding linear subspaces of $K^m$ for various $m$.

But this is doable, by the following basic and general fact:
\begin{lemma}\label{linear-codes}
Let $K$ be any field.  Let $V$ be a subspace of $K^n$.  Then $V$ can be coded by a tuple in $K$, and $V$ and $K^n/V$ have $\ulcorner V \urcorner$-definable bases.
\end{lemma}
\begin{proof}
Let $m = \dim V$.  By linear algebra, there is some coordinate projection $\pi : K^n \to K^m$ such that the restriction of $\pi$ to $V$ is an isomorphism $V \to K^m$.  Then the preimage of the standard basis under this isomorphism is a $\ulcorner V \urcorner$-definable basis for $V$.  This basis is a code for $V$.  Meanwhile, if we push the standard basis of $K^n$ forward to $K^n/V$, then some subset of this will be a basis for $K^n/V$, and will be definable over the parameters (such as $\ulcorner V \urcorner$) that were used to interpret the set $K^n/V$.
\end{proof}

For the case of ACVF, the coding of definable types will be done similarly.  But in addition to coding subspaces of $K^n$, we will also need to code definable ways of turning $K^n$ into a valued $K$-vector space.  The third condition of Theorem~\ref{ei-criterion} will be verified \emph{using} the coding of definable types.

\section{Elimination of imaginaries in ACVF}
In this section, we prove that ACVF has elimination of imaginaries in the sorts $K$, $R_{n,\ell}$, by applying Theorem~\ref{ei-criterion}.  Recall that we are referring to these as the geometric sorts.  We say that an object has a \emph{geometric code} if it has a code in these sorts.

\subsection{Coding modules}
Recall that $R_{n,\ell}$ is the set of pairs $(\Lambda,V)$ where $\Lambda$ is a lattice in $K^n$ and $V$ is an $\ell$-dimensional subspace of $\res \Lambda := \Lambda \otimes_{\mathcal{O}} k = \Lambda / \mathfrak{M} \Lambda$.

For fixed $\Lambda$, the poset of $k$-subspaces of $\res \Lambda$ is isomorphic to the poset of $\mathcal{O}$-submodules between $\mathfrak{M} \Lambda$ and $\Lambda$.  Moreover, $\ell$-dimensional subspaces correspond exactly to $\mathcal{O}$-submodules isomorphic to $\mathcal{O}^\ell \times \mathfrak{M}^{n - \ell}$.  So we could
equivalently define $R_{n,\ell}$ to be the set of all $\mathcal{O}$-submodules of $K^n$ isomorphic to $\mathcal{O}^\ell \times \mathfrak{M}^{n - \ell}$.  Under this identification,
\[ \bigcup_{\ell = 0}^n R_{n,\ell}\]
is the space of all open bounded definable $\mathcal{O}$-submodules of $K^n$, by Theorem~\ref{module-classification}.

In section \ref{module-classify}, we saw that $Mod_n$, the set of \emph{all} definable submodules of $K^n$, is interpretable.  We now show that $Mod_n$ can be embedded into the geometric sorts.
\begin{lemma}\label{code-modules}
If $M \le K^n$ is a definable submodule of $K^n$, then $M$ has a geometric code.
\end{lemma}
\begin{proof}
Let $V^+$ be the $K$-span of $N$, and let $V^-$ be the maximal $K$-subspace of $K^n$ contained in $N$.  By Lemma~\ref{linear-codes}, the subspaces $V^+$ and $V^-$ can be coded by a tuple $c$ from $K$, and the quotient $V^+/V^-$ has a $c$-definable identification with $K^m$, for some $m$.  Then $N$ is interdefinable over $c$ with the image of $N/V^-$ in $K^m$.  But this image will be an open bounded definable submodule, so it is an element of $R_{m,\ell}$ for some $\ell$.
\end{proof}

\subsection{Coding definable types}\label{sec:contribution}
A definable type in $K^n$ induces an ideal $I$ in $K[X_1,\ldots,X_n]$ together with the structure of a valued $K$-vector space on the quotient $K[\vec{X}]/I$.  By quantifier elimination in the one-sorted language, these data completely determine the type.  So the problem of finding codes for definable types reduces to the (easy) problem of coding subspaces, and the problem of coding valued vector space structures on $K$-vector spaces.\footnote{We will be more explicit in the proof of Theorem~\ref{code-definable} below.}

At the risk of being overly pedantic\ldots
\begin{definition}
Let $V$ be a $K$-vector space.  A \emph{VVS structure} on $V$ is a binary relation $R$ on $V$ such that there is a valued $K$-vector space structure $(V,\Gamma(V),\ldots)$ on $V$ for which $xRy \iff \val(x) \le \val(y)$.
\end{definition}
The VVS structures on $V$ are essentially the distinct ways of turning
$V$ into a valued $K$-vector space.  Two valued $K$-vector spaces $W$
and $W'$ with the same underlying vector space $V$ yield the same VVS
structure on $V$ iff they are identical up to an isomorphism of value
groups $\Gamma(W) \cong \Gamma(W')$.

If $V$ is a definable $K$-vector space, it makes sense to say that a VVS structure $R$ is ``definable,'' meaning that $R$ is a definable subset of $V \times V$.  If $R$ is definable, then $\Gamma(V)$ is interpretable and the map $\val : V \to \Gamma(V)$ and the action of $\Gamma(K)$ on $\Gamma(V)$ are all definable.

\begin{theorem}\label{coding-ways}
Let $\tau$ be (the code for) a definable VVS structure on $K^m$.  Then $\tau$ is interdefinable with an element of the geometric sorts.
\end{theorem}
\begin{proof}
Let $V$ be the associated valued $K$-vector space.  So $K^m$ is the underlying vector space of $V$ and $\tau$ is a code for the relation $\val(x) \le \val(y)$.  The set $\Gamma(V)$ is $\tau$-interpretable, as $K^m \setminus 0$ modulo the equivalence relation $\val(x) \le \val(y) \wedge \val(y) \le \val(x)$.

By Remark~\ref{few-orbits}, $\Gamma(V)$ consists of finitely many orbits under $\Gamma(K)$.

If there was only one orbit, and if there was a canonical identification of $\Gamma(V)$ with $\Gamma(K)$, we could proceed as follows: let $B$ be the closed ball around 0 with valuative radius 0.  The other closed balls around 0 are all of the form $\alpha B$, for $\alpha \in K$.
The set $B$ is a definable $\mathcal{O}$-submodule of $K^m$, so it has a geometric code.  It determines $\tau$, however, because $\val(x) \le \val(y)$ if and only if every closed ball containing 0 and $x$ contains $y$.  So $\tau$ and $\ulcorner B\urcorner$ would be interdefinable.

In general we have several orbits.  The first order of business is finding a $\tau$-definable element in each one:
\begin{claim}
Each orbit of $\Gamma(K)$ on $\Gamma(V)$ contains a $\tau$-definable element.
\end{claim}
\begin{claimproof}
For $x, y \in \Gamma(V)$, let $x \gg y$ indicate that $\Gamma(K) + x > \Gamma(K) + y$.  Let $x \sim y$ indicate that $x \not \gg y$ and $y \not \gg x$.  This is an equivalence relation.  Each orbit of $\Gamma(K)$ is in one $\sim$-equivalence class, so there are finitely many $\sim$-equivalence classes.  Each $\sim$-equivalence class is a convex subset of the linear order $\Gamma(V)$.  Let $C_1 > C_2 > \ldots > C_\ell$ be the distinct $\sim$-equivalence classes sorted in order from most positive to most negative.

For $0 \le i \le \ell$, let $V_i$ be the set of $v \in V$ such that $\val(v) \in C_j$ for some $j \le i$.  Each $V_i$ is a $K$-vector space, yielding an ascending filtration
\[ 0 = V_0 \subseteq V_1 \subseteq \cdots \subseteq V_\ell = V.\]
On $V_i \setminus V_{i-1}$, the function $\val$ lands in $C_i$ and factors through the quotient $V_i / V_{i-1}$, by the ultrametric inequality.

The equivalence relation $\sim$ is $\tau$-definable.  Since there are finitely many equivalence classes and they are totally ordered, each $C_i$ is $\tau$-definable.  Consequently, each $V_i$ is $\tau$-definable.  By Lemma~\ref{linear-codes}, each quotient space $V_i / V_{i-1}$ has a $\tau$-definable basis.  In particular, there is a $\tau$-definable non-zero vector in $V_i / V_{i-1}$.  Taking its valuation, we get a $\tau$-definable element of $C_i$.  We have shown:
\begin{equation}
\text{Each $C_i$ contains a $\tau$-definable element.} \label{halfway}
\end{equation}

Next, for $x, y \in \Gamma(V)$ let $x \approx y$ indicate that $x + \gamma > y > x - \gamma$ for all positive $\gamma \in \Gamma(K)$.  This is again a $\tau$-definable equivalence relation.  If $x$ and $y$ are in the same orbit, but are not equal, then $x \not \approx y$.  Consequently, each $\approx$-equivalence class contains at most one point from each orbit, so each $\approx$-equivalence class is finite.  This implies that if $x \approx y$, then $x$ and $y$ are interalgebraic over $\tau$.  In light of the total ordering, they are actually inter\emph{definable} over $\tau$.

Let $x$ be an arbitrary element of $\Gamma(V)$.  We will show that the orbit $\Gamma(K) + x$ contains a $\tau$-definable element.  By (\ref{halfway}), $x \sim y$ for some $\tau$-definable element $y$.  The set
\[ \{\gamma \in \Gamma(K) : \gamma + x \le y\}\]
is non-empty, because $x \not \gg y$, and it is bounded above, because $y \not \gg x$.  It is also definable, so it has a supremum $\gamma_0$, by o-minimality of $\Gamma(K)$.  Then $\gamma_0 + x \approx y$.  The element $\gamma_0 + x$ is interdefinable over $\tau$ with the $\tau$-definable element $y$, so it is itself $\tau$-definable.
\end{claimproof}
Given the claim, let $\gamma_1, \ldots, \gamma_n$ be a set of $\tau$-definable orbit representatives.  Let $B_i = \{v \in K^m : \val(v) \ge \gamma_i\}$.  Each $B_i$ is a $\tau$-definable $\mathcal{O}$-submodule of $K^m$, i.e., an element of $Mod_m$.  The closed balls of $V$ containing 0 are exactly the sets of the form $\alpha B_i$ for $1 \le i \le n$ and $\alpha \in K$.  The family of closed balls containing 0 is enough to determine the VVS structure, so $\tau$ is interdefinable with the tuple $(B_1,\ldots,B_n)$.  But by Lemma~\ref{code-modules}, each $B_i$ has a geometric code.
\end{proof}

\begin{theorem}\label{code-definable}
Let $p(x)$ be a definable type in $K^n$.  Then $p(x)$ has a code in the geometric sorts.
\end{theorem}
\begin{proof}
For each $d$, let $V_d$ be the space of polynomials in $K[X_1,\ldots,X_n]$ of degree $\le d$.  This is a finite dimensional definable $K$-vector space with a 0-definable basis.  Let $I_d$ be the set of $Q(X) \in V_d$ such that the formula $Q(x) = 0$ is in $p(x)$.  Let $R_d$ be the set of pairs $(Q_1(X),Q_2(X))$ in $V_d \times V_d$ such that the formula $\val(Q_1(x)) \le \val(Q_2(x))$ is in $p(x)$.  Then $I_d$ is a subspace of $V_d$, and $R_d$ induces a definable VVS structure on the quotient space $V_d/I_d$.  Quantifier elimination in the one-sorted language implies that $p$ is completely determined by the collection of all $I_d$'s and $R_d$'s for $d < \omega$.  By Lemma~\ref{linear-codes}, we can find codes $\ulcorner I_d \urcorner$ in the home sort for the $I_d$'s.  After naming these codes, each quotient space $V_d/I_d$ has a definable basis, and can be definably identified with some power of $K$.  Then each $R_d$ is interdefinable with a definable VVS structure on a power of $K$.  By Theorem~\ref{coding-ways}, these VVS structures have codes $c_d$ in the geometric sorts.  Now the union of all the $\ulcorner I_d \urcorner$'s and $c_d$'s is a code for $p$.
\end{proof}

\subsection{Coding finite sets}\label{sec:horror}
In this section, we show that finite sets of tuples from the geometric
sorts can be coded in the geometric sorts.  We make use of the
existence of geometric codes for definable types.  For a more
elementary but more complicated approach, see Proposition 3.4.1 in
\cite{HHM}.

\begin{definition}
If $X$ is some set, $\Sym^n X$ will denote the $n$-fold symmetric
product of $X$, that is, $X^n$ modulo the action of the $n$th
symmetric group.  The natural map $X^n \to \Sym^n X$ will be denoted
$\sigma$, so that
\[ \sigma(x_1,\ldots,x_n) = \sigma(y_1,\ldots,y_n)\]
if and only if there is a permutation $\pi$ of $n$ such that $x_i =
y_{\pi(i)}$ for $i = 1,\ldots,n$.
\end{definition}

\begin{definition}
  A 0-definable map $\pi : X \to Y$ has \emph{definable lifting} if
  for every $b \in Y$ there is a $b$-definable type $p_b$ in
  $\pi^{-1}(b)$.  (In particular, $\pi$ must be surjective.)  Say that
  $\pi$ has \emph{generically stable lifting} if it has definable
  lifting and $p_b$ can be taken to be generically stable.
\end{definition}

In both cases, we can easily modify the map $b \mapsto p_b$ to be
automorphism equivariant, so that if $\sigma \in
\Aut(\mathbb{M}/\emptyset)$, then $p_{\sigma(b)} = \sigma(p_b)$ for
every $b$.  Conversely, if there is an automorphism equivariant map $b
\mapsto p_b$ from elements of $Y$ to definable (resp. generically
stable) types in $X$, such that $p_b$ is in $\pi^{-1}(b)$, then $\pi$
has definable (resp. generically stable) lifting---the automorphism
invariance ensures that $p_b$ is $b$-definable.

If $\pi : X \to Y$ has definable lifting, and $q$ is a $C$-definable
type in $Y$ for some parameters $C$, then $q = \pi_*p$ for some
$C$-definable type $p$ in $X$.  Indeed, if $M$ is a model containing
$C$, and $b$ realizes $q | M$, and $a \in \pi^{-1}(b)$ realizes $p_b |
Mb$, then $a \forkindep^{\textrm{def}}_{Cb} M$ and $b
\forkindep^{\textrm{def}}_C M$, so $a \forkindep^{\textrm{def}}_C M$
by left-transitivity of $\forkindep^{\textrm{def}}$.  Thus $a \models
p|M$ for some $C$-definable type $p$.  By definition of pushforward,
$\pi(a) = b$ realizes $\pi_* p | M$.  Then the $C$-definable types
$\pi_* p$ and $q$ have the same restriction to a model containing $C$,
so they are equal.

From this, we draw the following conclusion:
\begin{observation}
If $\pi : X \to Y$ has definable lifting, and definable types in $X$
have codes in the geometric sorts, then so do definable types in $Y$.
\end{observation}
Indeed, let $q$ be a definable type in $Y$.  Then $q$ is $\ulcorner q
\urcorner$-definable, so $q = \pi_* p$ for some $\ulcorner q
\urcorner$-definable type $p$.  But then $p$ and $q$ are
interdefinable, and $p$ has a geometric code.

We make two other useful observations:
\begin{observation}\label{obs:products}
  If $\pi : X \to Y$ and $\pi' : X' \to Y'$ both have definable
  lifting, then so does the product map $\pi \times \pi' : X \times X'
  \to Y \times Y'$.  Indeed, if $f$ and $f'$ are the
  automorphism-equivariant maps witnessing definable lifting, then
  $(b,b') \mapsto f(b) \otimes f'(b')$ witnesses definable lifting for
  the product map $\pi \times \pi'$.  The same statement holds with
  ``generically stable lifting'' in place of ``definable lifting,'' by Theorem~\ref{gs}(a).
\end{observation}
\begin{observation}\label{obs:gen-es}
If $\pi : X \to Y$ has generically stable lifting, then so does
$\Sym^n \pi : \Sym^n X \to \Sym^n Y$.
Indeed, suppose that $b \mapsto p_b$ is the automorphism equivariant
map from elements of $Y$ to generically stable types in $X$.  Then
\[ \sigma(b_1,\ldots,b_n) \mapsto \sigma_*(p_{b_1} \otimes \cdots \otimes p_{b_n})\]
is a well-defined automorphism-equivariant map from elements of
$\Sym^n Y$ to generically stable types in $\Sym^n X$, witnessing
generically stable lifting for $\Sym^n \pi$.  Generic stability ensures that
\[ \sigma_*(p_{b_1} \otimes \cdots \otimes p_{b_n}) = \sigma_*(p_{b_{\pi(1)}} \otimes \cdots \otimes p_{b_{\pi(n)}})\]
for any $b_1,\ldots,b_n$ and any permutation $\pi$ of
$\{1,\ldots,n\}$.
\end{observation}

\begin{lemma}\label{better-version}
  Definable lifting holds for the map $\Sym^n(K^m \times \Oo^\ell) \to
  \Sym^n(K^m \times k^\ell)$, induced by the componentwise residue
  $\Oo^\ell \to k^\ell$.
\end{lemma}
\begin{proof}
  First note that the residue map $\res : \mathcal{O} \to k$ has
  definable lifting: to each residue $\alpha \in k$ we associate the
  generic type of the open ball $\res^{-1}(\alpha)$.  By
  Observation~\ref{obs:products}, the componentwise residue map $\Oo^n
  \to k^n$ has definable lifting.  It follows that the map $\Sym^n \Oo
  \to \Sym^n k$ has definable lifting.  Indeed, there is a commutative
  diagram
  \begin{equation*}
    \xymatrix{ \Sym^n \Oo \ar[d] \ar[r]^\sim  & \Oo^n \ar[d] \\ \Sym^n k \ar[r]^\sim & k^n}
  \end{equation*}
  in which the vertical maps are induced by the residue map $\Oo \to
  k$, and the horizontal maps are induced by the elementary symmetric
  polynomials.  These horizontal maps are bijections because $k$ and
  $K$ are algebraically closed and $\Oo$ is integrally closed.  So
  definable lifting holds for $\Sym^n \Oo \to \Sym^n k$.

  Now consider the map $\Sym^n(K^m \times \Oo^\ell) \to \Sym^n(K^m
  \times k^\ell)$.  Sweeping the set/multiset distinction under the
  rug, let $S$ be a finite subset of $K^m \times k^\ell$.  Let $T
  \subseteq k$ be the set of all elements of $k$ appearing in $S$.  By
  definable lifting of $\Sym^i \Oo \to \Sym^i k$, there is a finite
  subset $T' \subseteq \Oo$ such that $\ulcorner T' \urcorner
  \forkindep^{\textrm{def}}_{\ulcorner S \urcorner} \ulcorner S
  \urcorner$ and such that $T'$ lifts $T$, meaning that the residue
  map $\Oo \to k$ restricts to a bijection $T' \to T$.  Let $f : T \to
  T'$ be the inverse, a partial section of the residue map.  Let $S'
  \subseteq K^m \times \Oo^\ell$ be obtained by applying $f$
  componentwise to $S$.  Then $S'$ maps onto $S$ via the componentwise
  residue map $K^m \times \Oo^\ell \to K^m \times k^\ell$.  Moreover,
  $S'$ is definable over $\ulcorner S \urcorner \ulcorner T'
  \urcorner$, and so
  \begin{equation*}
    \ulcorner S' \urcorner \forkindep^{\textrm{def}}_{\ulcorner S \urcorner} \ulcorner S \urcorner.
  \end{equation*}
  Therefore $\ulcorner S' \urcorner \in \Sym^n(K^m \times \Oo^\ell)$
  lifts $\ulcorner S \urcorner \in \Sym^n(K^m \times k^\ell)$, and
  realizes an $\ulcorner S \urcorner$-definable type.
\end{proof}
\begin{corollary}\label{some-cor}
  Definable types in $\Sym^n(K^m \times k^\ell)$ have geometric codes.
\end{corollary}
\begin{proof}
  The set $\Sym^n(K^m \times \Oo^\ell)$ embeds into $\Sym^n(K^{m +
    \ell})$, which 0-definably embeds into some power $K^N$ by
  elimination of imaginaries in ACF.  By Theorem~\ref{code-definable},
  definable types in $K^N$ have geometric codes.  Therefore, definable
  types in $\Sym^n(K^m \times \Oo^\ell)$ have geometric codes.  By
  Observation~\ref{obs:products} and Lemma~\ref{better-version}, the
  Corollary follows.
\end{proof}

\begin{theorem}\label{finite-codes}
Let $G$ be any geometric sort.  Then elements of $\Sym^n G$ have
geometric codes.
\end{theorem}
\begin{proof}
We claim that there is some map $G' \to G$ with generically stable
lifting, such that $G'$ embeds (0-definably) into a product $K^m
\times k^\ell$.  Assuming this is true, we get a map
\[ \Sym^n(G') \to \Sym^n(G)\]
with generically stable (hence definable) lifting, and $\Sym^n(G')$
embeds into $\Sym^n(K^m \times k^\ell)$.  Corollary~\ref{some-cor} ensures
that definable types in $\Sym^n(G')$ have geometric codes, so by
definable lifting, definable types in $\Sym^n(G)$ have geometric
codes.  In particular, looking at constant types in $\Sym^n(G)$, we
see that elements of $\Sym^n(G)$ have geometric codes.

It remains to find $G' \to G$.  The property of generically stable
lifting is closed under taking products
(Observation~\ref{obs:products}), and $G$ is a product of $K$'s and
$R_{n,\ell}$'s, so it suffices to consider the case $G = K$ or $G =
R_{n,\ell}$.  For $G = K$, we take the identity map $G' = K \to K =
G$.  The case $G = R_{n,\ell}$ remains.  Let $\widetilde{R_{n,\ell}}$
be the set of triples $(\vec{b},\Lambda,V)$, where $\Lambda$ is a
lattice in $K^n$, $\vec{b}$ is a lattice basis, and $V$ is an
$\ell$-dimensional $k$-subspace of $\res \Lambda := \Lambda /
\mathfrak{M}\Lambda$.  Then the canonical map
\[ \widetilde{R_{n,\ell}} \to R_{n,\ell}\]
\[ (\vec{b},\Lambda,V) \mapsto (\Lambda,V)\]
has generically stable lifting.  Indeed, given $\Lambda$, any
realization of $p_\Lambda^{\otimes n}$ will be a basis of $\Lambda$,
where $p_\Lambda$ is the generic type of $\Lambda$.

Moreover, we can embed $\widetilde{R_{n,\ell}}$ into a
product of $K$'s and $k$'s.  Let $Gr_{n,\ell}$ denote the set of
$\ell$-dimensional $k$-subspaces of $k^n$.  Then there is a
0-definable map
\[ \widetilde{R_{n,\ell}} \to K^{n^2} \times Gr_{n,\ell}\]
\[ (\vec{b},\Lambda,V) \mapsto (\vec{b},W),\]
where $W$ is the image of $V$ under the identification of $\res
\Lambda$ with $k^n$ induced by the basis $\vec{b}$.  This map is an
injection, and $Gr_{n,\ell}$ can be embedded in a power of $k$ by
algebraic geometry, or elimination of imaginaries in ACF.
\end{proof}

\subsection{Putting everything together}

\begin{theorem}\label{main-theorem}
ACVF has elimination of imaginaries in the geometric sorts, i.e., in $K$ and the $R_{n,\ell}$.
\end{theorem}
\begin{proof}
This follows by Theorem~\ref{ei-criterion}.  The second condition is Theorem~\ref{code-definable}.  The third condition is Theorem~\ref{finite-codes}.  The first condition of Theorem~\ref{ei-criterion} can be verified as follows: Let $D$ be a one-dimensional definable set.  Then $D$ can be written as a disjoint union of $\acl^{eq}(\ulcorner D \urcorner)$-definable ``swiss cheeses.''  If $B \setminus (B_1 \cup \cdots \cup B_n)$ is one of these swiss cheeses, then the generic type $p_B$ of $B$ is in $B \setminus (B_1 \cup \cdots \cup B_n)$, hence in $D$.  Since $B$ is $\acl^{eq}(\ulcorner D \urcorner)$-definable, so is $p_B$.
\end{proof}

\section{Reduction to the standard geometric sorts}\label{coda}
Haskell, Hrushovski, and Macpherson showed that ACVF has elimination of imaginaries in the sorts $K, S_n, T_n$.  To deduce this result from Theorem~\ref{main-theorem}, we need to code the $R_{m,\ell}$ sorts into the $S_n$ and $T_n$.  This is done in 2.6.4 of \cite{HHM}, but for the sake of completeness, we quickly recall the arguments here.  Recall that $S_n$ is the set of lattices in $K^n$, and $T_n$ is the union of $\res \Lambda$ as $\Lambda$ ranges over $S_n$.

First of all, we can easily code the $R_{m,\ell}$ in terms of $R_{n,0} (= S_n)$ and $R_{n,1}$ (which is roughly a projectivized version of $T_n$).  Indeed, if $\Lambda$ is a lattice in $K^n$, and $V$ is an $\ell$-dimensional subspace in $\res(V)$, then $V$ can be coded by a one-dimensional subspace (namely $\bigwedge^\ell V$) in
\[ \bigwedge^\ell \res(\Lambda) = \res(\bigwedge^\ell \Lambda),\]
and $\bigwedge^\ell \Lambda$ is a lattice in $\bigwedge^\ell K^n$.
So to code an element of $R_{n,\ell}$, we can use the underlying lattice in $R_{n,0}$, and then an element in $R_{N,1}$, where $N = \dim \bigwedge^\ell K^n$.

Now to code an element of $R_{n,1}$ in terms of the $S_n$ and $T_n$, we proceed by induction.  Let $\Lambda$ be a lattice in $K^n$.  Let $\pi$ be the projection onto the first coordinate.  Then $\pi(\Lambda)$ is free\footnote{Finitely generated torsion-free $\mathcal{O}$-modules are always free.}, so we have a split exact sequence
\[ 0 \to \Lambda' \to \Lambda \to \pi(\Lambda) \to 0\]
where $\Lambda'$ is a lattice in $K^{n-1}$.  Since this sequence is split exact, it remains split exact after tensoring with $k$.  So
\[ 0 \to \res(\Lambda') \to \res(\Lambda) \to \res(\pi(\Lambda)) \to 0\]
is exact.  Let $V$ be a one dimensional subspace of $\res(\Lambda)$.  If $V$ sits inside $\res(\Lambda')$, then $V$ is interdefinable with a one-dimensional subspace of $\res(\Lambda')$, so can be coded in the true geometric sorts by induction.

Otherwise, $V$ maps isomorphically onto the one-dimensional $k$-space $\res(\pi(\Lambda))$.  Then to code $V$, it suffices to code the inverse map $\res(\pi(\Lambda)) \to V \hookrightarrow \res(\Lambda)$.  But because all the $\mathcal{O}$-modules in sight are free,
\[ \Hom_k(\res(\pi(\Lambda)),\res(\Lambda)) = \Hom_k(\pi(\Lambda) \otimes k, \Lambda \otimes k) = \Hom_\mathcal{O}(\pi(\Lambda),\Lambda) \otimes k.\]
And $\Hom_\mathcal{O}(\pi(\Lambda),\Lambda)$ is a lattice in $\Hom_K(K,K^n) \cong K^n$.  So a map from $\res(\pi(\Lambda))$ to $\res(\Lambda)$ can be coded by an element of $T_n$.

\begin{acknowledgment}
The author would like to thank Tom Scanlon, Ehud Hrushovski, and the participants in the 2014 Berkeley Model Theory Seminar.
{\tiny Also, this material is based upon work supported by the National Science Foundation Graduate Research Fellowship under Grant No. DGE 1106400.  Any opinions, findings, and conclusions or recommendations expressed in this material are those of the author and do not necessarily reflect the views of the National Science Foundation.}
\end{acknowledgment}

\bibliographystyle{plain} \bibliography{mybib}{}

\end{document}